\documentclass{article}
\usepackage{amssymb,amsthm,amsmath}

\newtheorem{thm}{Theorem}[section]
\newtheorem{prop}[thm]{Proposition}

\newtheorem{cor}[thm]{Corollary}
\newtheorem{conj}[thm]{Conjecture}

\def\Ex{\mathbb E}
\def\er{\mathbb R}
\def\Pr{\mathbb P}
\def\ve{\varepsilon}

\def\Cov{\mathrm{Cov}}
\def\cale{{\mathcal E}}

\theoremstyle{definition}

\newtheorem*{xrem}{Remark}

\begin{document}



\title{Weak and strong moments of random vectors}

\author{Rafa{\l} Lata{\l}a}
\date{}

\maketitle

\begin{abstract}
We discuss a conjecture about comparability of weak and strong moments of log-concave 
random vectors and show the conjectured inequality for unconditional vectors in 
normed spaces with a bounded cotype constant.
\end{abstract}

\section{Introduction} 

Let $X$ be a random vector with values in some normed space $(F,\|\ \|)$. The question we will discuss
is how to estimate $\|X\|_p=(\Ex\|X\|^p)^{1/p}$ for $p\geq 1$. Obviously
$\|X\|_p\geq \|X\|_1=\Ex\|X\|$ and for any continuous linear functional $\varphi$ on $F$ with
$\|\varphi\|_{*}\leq 1$ we have $\|X\|_p\geq (\Ex|\varphi(X)|^p)^{1/p}$. It turns out that in some
situations one may reverse these obvious estimates and show that for an absolute constant $C$ 
and any $p\geq 1$,
\[
(\Ex\|X\|^p)^{1/p}\leq C\Big(\Ex\|X\|+\sup_{\|\varphi\|_{*}\leq 1}(\Ex|\varphi(X)|^p)^{1/p}\Big).
\]
This is for example the case when $X$ has Gaussian or product exponential distribution.
In this note we will concentrate on the more general case of log-concave vectors.

A measure $\mu$ on $\er^n$ is called logarithmically concave (log-concave in short) if 
for any compact nonempty sets $A,B\subset \er^n$ and $\lambda\in (0,1)$,
\[
\mu(\lambda A+(1-\lambda)B)\geq \mu(A)^{\lambda}\mu(B)^{1-\lambda}.
\]
By the result of  Borell \cite{Bo} a measure $\mu$ on $\er^n$ with full dimensional support 
is log-concave  if and only if it is absolutely continuos with respect to the Lebesgue measure and
has a density of the form $e^{-f}$, where 
$f\colon \er^n\rightarrow (-\infty,\infty]$
is a convex function. Log-concave measures are frequently
studied in convex geometry, since by the Brunn-Minkowski inequality  uniform distributions
on convex bodies as well as their lower dimensional marginals are log-concave. 
In fact the class of log-concave measures on $\er^n$ is the smallest class of 
probability measures closed under linear transformation 
and weak limits that contains uniform distributions on convex bodies. Vectors with logaritmically
concave distributions are called log-concave.

In the sequel we discuss the following conjecture posed in a stronger form in \cite{LW} 
about the comparison of strong
and weak moment for log-concave vectors.

\begin{conj}
\label{ws_conj}
For any $n$ dimensional log-concave random vector and any norm $\|\ \|$ on $\er^n$ we have
for $1\leq p<\infty$,
\begin{equation}
\label{weak_strong}
(\Ex\|X\|^p)^{1/p}\leq C_1\Ex\|X\|+C_2\sup_{\|\varphi\|_{*}\leq 1}(\Ex|\varphi(X)|^p)^{1/p},
\end{equation}
where $C_1$ and $C_2$ are absolute constants.
\end{conj}

In Section 2 we gather known results about validity of \eqref{weak_strong} in special cases. Section 3 is
devoted to the unconditional vectors. In particular we show that Conjecture \ref{ws_conj} is satisfied
under additional assumption of unconditionality of $X$ and bounded cotype constant of 
the underlying normed space. 

\subsection*{Notation} Let $(\ve_i)$ be a Bernoulli sequence, i.e.\
a sequence of independent symmetric variables taking values $\pm 1$. We assume that 
$(\ve_i)$ are independent of other random variables.

By $(\cale_i)$ we denote a sequence of independent symmetric exponential random variables
with variance 1 (i.e. the density $2^{-1/2}\exp(-\sqrt{2}|x|)$). We set
$\cale=\cale^{(n)}=(\cale_1,\ldots,\cale_n)$ for an $n$-dimensional random vector with product 
exponential distribution and identity covariance matrix.

By $\langle\cdot,\cdot\rangle$ we denote the standard scalar
product on $\er^n$ and by $(e_i)$ the standard basis of $\er^n$.
 We set $B_p^n$ for a unit ball in ${\ell}^n_p$, i.e.\
$B_p^n=\{x\in\er^n\colon \|x\|_p\leq 1\}$. For a random variable $Y$ and $p>0$ we write
$\|Y\|_p=(\Ex|Y|^p)^{1/p}$.

We write $C$ (resp.\ $C(\alpha)$) to denote universal constants (resp.\ constants 
depending only on parameter $\alpha$). Value of a constant $C$ may differ at each
occurence. 

\section{Known results} Since any norm on $\er^n$ may be approximated by a supremum
of exponential number of functionals we get

\begin{prop}[see  {\cite[Proposition 3.20]{LW}}]
For any $n$-dimensional random vector $X$ inequality \eqref{weak_strong} holds for $p\geq n$
with $C_1=0$ and $C_2=10$.
\end{prop}

It is also easy to reduce Conjecture \ref{ws_conj} to the case of symmetric vectors.

\begin{prop}
\label{symmetrization}
Suppose that \eqref{weak_strong} holds for all symmetric $n$-dimensional log-concave vectors $X$.
Then it is also satisfied with constants $4C_1+1$ and $4C_2$ by all log-concave vectors $X$. 
\end{prop}

\begin{proof}
Assume first that $X$ has a log-concave distribution and $\Ex X=0$. Let $X'$ be an independent copy of $X$,
then $X-X'$ is symmetric and log-concave. Moreover for $p\geq 1$,
\begin{align*}
(\Ex\|X\|^p)^{1/p}&=(\Ex\|X-\Ex X'\|^p)^{1/p}\leq (\Ex\|X-X'\|^p)^{1/p},
\\
\Ex\|X-X'\|&\leq \Ex\|X\|+\Ex\|X'\|=2\Ex\|X\|
\end{align*}
and for any functional $\varphi$,
\[ 
(\Ex|\varphi(X-X')|^p)^{1/p}\leq (\Ex|\varphi(X)|^p)^{1/p}+(\Ex|\varphi(X')|^p)^{1/p}=2
(\Ex|\varphi(X)|^p)^{1/p}
\] 
Hence \eqref{weak_strong} holds for $X$ with constant $2C_1$ and $2C_2$.

If $X$ is arbitrary log-concave then $X-\Ex X$ is log-concave with mean zero. We have for any $p\geq 1$,
\[
(\Ex\|X\|^p)^{1/p}\leq (\Ex\|X-\Ex X\|^p)^{1/p}+\Ex\|X\|, \quad
\Ex\|X-\Ex X\|\leq 2\Ex\|X\|
\]
and for any functional $\varphi$,
\[
(\Ex|\varphi(X-\Ex X)|^p)^{1/p}\leq (\Ex|\varphi(X)|^p)^{1/p}+|\varphi(\Ex X)|\leq
2(\Ex|\varphi(X)|^p)^{1/p}.
\]
\end{proof}

\begin{xrem}
Estimating  $\|X\|_p$ is strictly connected with bounding tails of $\|X\|$. Indeed by
Chebyshev's inequality we have
\[
\Pr(\|X\|\geq e\|X\|_p)\leq e^{-p}
\]
and by the Paley-Zygmund inequality and the fact that $\|X\|_{2p}\leq C\|X\|_p$ for $p\geq 1$ we get
\[
\Pr\Big(\|X\|\geq \frac{1}{C}\|X\|_p\Big)\geq \min\Big\{\frac{1}{C},e^{-p}\Big\}.
\]
\end{xrem}

Gaussian concentration inequality easily implies \eqref{weak_strong} for Gaussian vectors $X$
(see for example Chapter 3 of \cite{LT}). For Rademacher sums comparability of weak and strong moments was established by Dilworth and Montgomery-Smith \cite{DMS}.
More general statement was shown in \cite{L3}.

\begin{thm}
\label{ind_sum}
Suppose that $X=\sum_{i}v_i\xi_i$, where $v_i\in F$ and $\xi_i$ are independent symmetric
r.v's with logarithmically concave tails. Then for any $p\geq 1$ inequality \eqref{weak_strong} 
holds with absolute constants $C_1$ and $C_2$.
\end{thm}

This immediately implies

\begin{cor}
Conjecture \ref{ws_conj} holds  under additional assumption that coordinates of $X$
are independent. 
\end{cor}

\begin{proof}
We have $X=\sum_{i=1}^n e_iX_i$ with $X_i$ independent log-concave real random variables. It is enough to
notice that  variables $X_i$ have log-concave tails and in the symmetric case apply Theorem \ref{ind_sum}. General
independent case may be reduce to the symmetric one as in the proof of Proposition \ref{symmetrization}.
\end{proof}

The crucial tool in the proof of Theorem \ref{ind_sum} is the  Talagrand two-level concentration
inequality for the product exponential distribution \cite{Ta1}:
\[
\nu^n(A)\geq \frac{1}{2}\quad \Rightarrow\quad 1-\nu^n(A+\sqrt{t}B_2^n+tB_1^n)\leq e^{-t/C},\ t>0,
\]
where $\nu$ is the symmetric exponential distribution, i.e. $d\nu(x)=\frac{1}{2}\exp(-|x|)dx$.

In \cite{LW} more general concentration inequalities were investigated. For a probability measure $\mu$ 
on $\er^n$ define
\[
\Lambda_{\mu}(y)=\log\int e^{\langle y,z\rangle}d\mu(z),\quad
\Lambda_{\mu}^*(x)=\sup_{y}(\langle y,x\rangle -\Lambda_{\mu}(y))
\]
and
\[
B_{\mu}(t)=\{x\in\er^n\colon\ \Lambda_\mu(x)\leq t\}. 
\]
One may show that $B_{\nu^n}(t)\sim \sqrt{t}B_2^n+tB_1^n$.
Argument presented in {\cite[Section 3.3]{LW}} gives

\begin{prop}
Suppose that for some $\alpha\geq 1$ and
$\beta>0$  and any convex symmetric compact set $K\subset \er^n$ we have 
\begin{equation}
\label{conc}
\mu(K)\geq \frac{1}{2}\quad \Rightarrow\quad 1-\mu(\alpha K+B_{\mu}(t))\leq e^{-t/\beta},\ \mbox{ for all }t>0.
\end{equation}
Then inequality \eqref{weak_strong} holds with $C_1=\alpha$ and $C_2=C\beta$.
\end{prop} 

In \cite{LW} it was shown that concentration inequality \eqref{conc} holds with $\alpha=1$ for 
symmetric product log-concave measures and for uniform distributions on $B_r^n$ balls. This gives

\begin{cor}
Inequality \eqref{weak_strong} holds with $C_1=1$ and universal $C_2$ for uniform distributions
on $B_r^n$ balls $1\leq r\leq \infty$.
\end{cor}

Modification of Paouris'  proof \cite{Pa} of large deviation inequality for $\ell_2$ norm of 
isotropic log-concave vectors
shows that weak and strong moments are comparable in the Euclidean case (see \cite{ALLPT} for details):

\begin{thm}
If $X$ is a log-concave $n$-dimensional random vector then for any Euclidean norm $\|\ \|$ on $\er^n$ we have
\[
(\Ex\|X\|^p)^{1/p}\leq C\Big(\Ex\|X\|+\sup_{\|\varphi\|_*\leq 1}(\Ex|\varphi(X)|^p)^{1/p}\Big).
\]
\end{thm}

\section{Unconditional case} We say that a random vector $X=(X_1,\ldots,X_n)$ has {\em unconditional 
distribution} if the distribution of $(\eta_1X_1,\ldots,\eta_nX_n)$ is the same as $X$ for any choice
of signs $\eta_1,\ldots,\eta_n$. Random vector $X$ is called {\em isotropic} if it has identity
covariance matrix, i.e.\ $\Cov(X_i,X_j)=\delta_{i,j}$.

\begin{thm}
\label{uncond}
Suppose that 
$X$ is an $n$-dimensional
isotropic, unconditional, log-concave vector. Then for any norm $\|\ \|$ on $\er^n$ and $p\geq 1$,
\begin{equation}
\label{newdomination}
(\Ex\|X\|^p)^{1/p}\leq C\Big(\Ex\|\cale\|+\sup_{\|\varphi\|_{*}\leq 1}(\Ex|\varphi(X)|^p)^{1/p}\Big).
\end{equation}
\end{thm}

\begin{proof}
Let $T=\{t\in \er^n\colon \|t\|_{*}\leq 1\}$ be the unit ball in the  space $(\er^n,\|\ \|_{*})$
dual to $(\er^n,\|\ \|)$.
Then $\|x\|= \sup_{t\in T}\langle t,x\rangle$.
By the result of Talagrand \cite{Ta2} (see also \cite{Ta}) there exist subsets
$T_n\subset T$ and functions $\pi_n\colon T\rightarrow T_n$, $n=0,1,\ldots$ such that 
$\pi_n(t)\rightarrow t$ for all $t\in T$, $\# T_0=1$, $\# T_n\leq 2^{2^n}$ and 
\begin{equation}
\label{Talmaj}
\sum_{n=0}^{\infty}\|\langle \pi_{n+1}(t)-\pi_n(t),\cale \rangle\|_{2^n}\leq 
C\Ex\sup_{t\in T}\langle t,\cale\rangle=C\Ex\|\cale\|. 
\end{equation}
 
Let us fix $p\geq 1$ and choose $n_0\geq 1$ such that $2^{n_0-1}<2p\leq 2^{n_0}$. We have
\begin{equation}
\label{split}
\|X\|=\sup_{t\in T}\langle t,X\rangle\leq
\sup_{t\in T}|\langle \pi_{n_0}(t),X\rangle|
+\sup_{t\in T}\sum_{n=n_0}^{\infty}|\langle \pi_{n+1}(t)-\pi_n(t),X\rangle|.
\end{equation}

We get
\begin{align}
\notag
\Big(\Ex\sup_{t\in T}|\langle \pi_{n_0}(t),X\rangle|^p\Big)^{1/p}&\leq
\Big(\Ex\sum_{s\in T_{n_0}}|\langle s,X\rangle|^p\Big)^{1/p}\leq
(\# T_{n_0})^{1/p}\sup_{s\in T_{n_0}}(\Ex|\langle s,X\rangle|^p)^{1/p}
\\
\label{summand1}
&\leq 16\sup_{t\in T}(\Ex|\langle t,X\rangle|^p)^{1/p}
=16\sup_{\|\varphi\|_{*}\leq 1}(\Ex |\varphi(X)|^p)^{1/p}.
\end{align}

To estimate the last term in \eqref{split} notice that for $u\geq 16$ we have by Chebyshev's
inequality 
\begin{align*}
\Pr\bigg(\sup_{t\in T}&\sum_{n=n_0}^{\infty}|\langle \pi_{n+1}(t)-\pi_n(t),X\rangle|\geq
u\sup_{t\in T}\sum_{n=n_0}^{\infty}\|\langle \pi_{n+1}(t)-\pi_n(t),X \rangle\|_{2^n}\bigg)
\\
&\leq \Pr\Big(\exists_{n\geq n_0}\exists_{t\in T}\
|\langle \pi_{n+1}(t)-\pi_n(t),X\rangle|\geq u\|\langle \pi_{n+1}(t)-\pi_n(t),X \rangle\|_{2^n}\Big)
\\
&\leq \sum_{n=n_0}^{\infty}\sum_{s\in T_{n+1}}\sum_{s'\in T_n}
\Pr(|\langle s-s',X\rangle|\geq u\|\langle s-s',X \rangle\|_{2^n})
\leq \sum_{n=n_0}^{\infty}\#T_{n+1}\#T_n u^{-2^n}
\\
&\leq \sum_{n=n_0}^{\infty}\Big(\frac{8}{u}\Big)^{2^n}
\leq 2\Big(\frac{8}{u}\Big)^{2^{n_0}}\leq 2\Big(\frac{8}{u}\Big)^{2p}.
\end{align*}

Integrating by parts this gives
\begin{align}
\notag
\Big(\Ex\Big(\sup_{t\in T}&\sum_{n=n_0}^{\infty}|\langle \pi_{n+1}(t)-\pi_n(t),X\rangle|\Big)^p\Big)^{1/p}
\\
\notag
&\leq 
\sup_{t\in T}\sum_{n=n_0}^{\infty}\|\langle \pi_{n+1}(t)-\pi_n(t),X \rangle\|_{2^n}
\Big(16+\Big(2p\int_{0}^{\infty}u^{p-1}\Big(\frac{8}{u+16}\Big)^{2p}\Big)^{1/p}\Big)
\\
\label{summand2}
&\leq 32\sup_{t\in T}\sum_{n=n_0}^{\infty}\|\langle \pi_{n+1}(t)-\pi_n(t),X \rangle\|_{2^n}.
\end{align}

The result of Bobkov and Nazarov \cite{BN} gives 
\begin{equation}
\label{Bob_Naz}
\|\langle t, X\rangle \|_{r}\leq C\|\langle t, \cale\rangle \|_{r} \quad \mbox{ for any }
t\in \er^n \mbox{ and }r\geq 1.
\end{equation}
Thus the statement follows by \eqref{Talmaj}-\eqref{summand2}.
\end{proof}

\begin{xrem}
The only property of the vector $X$ that was used in the above proof was estimate \eqref{Bob_Naz}.
Thus inequality \eqref{newdomination} holds for all $n$-dimensional random vectors satisfying \eqref{Bob_Naz}. 
\end{xrem}

\begin{xrem}
Estimate \eqref{Bob_Naz} gives
$(\Ex|\varphi(X)|^p)^{1/p}\leq C(\Ex|\varphi(\cale)|^p)^{1/p}$ for any functional $\varphi$,
therefore Theorem \ref{uncond} is stronger than the estimate from \cite{L2}:
\[
(\Ex\|X\|^p)^{1/p}\leq C\Ex\|\cale\|^p\sim
C\Big(\Ex\|\cale\|+\sup_{\|\varphi\|_{*}\leq 1}(\Ex|\varphi(\cale)|^p)^{1/p}\Big).
\]
\end{xrem}

In some situation one may show that $\Ex\|\cale\|\leq C\Ex\|X\|$. This is the case of spaces with
bounded cotype constant.

\begin{cor}
\label{cotype}
Suppose that $2\leq q<\infty$, $F=(\er^n,\|\ \|)$ is a finite dimensional space with a
$q$-cotype constant bounded by $\beta<\infty$. Then for any $n$-dimensional unconditional, log-concave vector $X$  and $p\geq 1$,
\[
(\Ex\|X\|^p)^{1/p}\leq C(q,\beta)\Big(\Ex\|X\|+\sup_{\|\varphi\|_{*}\leq 1}(\Ex|\varphi(X)|^p)^{1/p}\Big),
\]
where $C(q,\beta)$ is a constant that depends only on $q$ and $\beta$.
\end{cor}

\begin{proof}
Applying diagonal transformation (and appropriately changing the norm) we may assume that $X$ is also
isotropic.

By the result of Maurey and Pisier \cite{MP} (see also Appendix II in \cite{MS}) one has
\[
\Ex\|\cale\|=\Ex\bigg\|\sum_{i=1}^n e_i\cale_i\bigg\|\leq C_1(q,\beta)\Ex\bigg\|\sum_{i=1}^n e_i\ve_i\bigg\|.
\]
By the unconditionality of $X$ and Jensen's inequality we get
\[
\Ex\|X\|=\Ex\bigg\|\sum_{i=1}^n e_i\ve_i|X_i|\bigg\|\geq \Ex\bigg\|\sum_{i=1}^n e_i\ve_i\Ex|X_i|\bigg\|.
\]
We have $\Ex|X_i|\geq \frac{1}{C}(\Ex|X_i|^2)^{1/2}=\frac{1}{C}$, therefore
\[
\Ex\|\cale\|\leq CC_1(q,\beta)\Ex\|X\|
\]
and the statement follows by Theorem \ref{uncond}.
\end{proof}

For general norm on $\er^n$ one has 
\[
\Ex\|\cale\|=\Ex\bigg\|\sum_{i=1}^n e_i\ve_i|\cale_i|\bigg\|\leq 
\Ex\sup_{i}|\cale_i|\Ex\bigg\|\sum_{i=1}^n e_i\ve_i\bigg\|
\leq C\log n\ \Ex\bigg\|\sum_{i=1}^n e_i\ve_i\bigg\|.
\]
This together with the similar argument as in the proof of Corollary \ref{cotype} gives the following.

\begin{cor}
For any $n$-dimensional unconditional, log-concave vector $X$, any norm $\|\ \|$ on $\er^n$ and
$p\geq 1$ one has
\[
(\Ex\|X\|^p)^{1/p}\leq C\Big(\log n\ \Ex\|X\|+\sup_{\|\varphi\|_{*}\leq 1}(\Ex|\varphi(X)|^p)^{1/p}\Big).
\]
\end{cor}

\subsection*{Acknowledgments}

Research of R. Lata{\l}a  was partially supported by the Foundation for Polish Science and
MNiSW grant N N201 397437.

\noindent
Institute of Mathematics, University of Warsaw\\
Banacha 2, 02-097 Warszawa, Poland\\
Institute of Mathematics, Polish Academy of Sciences\\
\'Sniadeckich 8, 00-956 Warszawa, Poland\\
E-mail: {\tt rlatala@mimuw.edu.pl}


\begin{thebibliography}{10}

\bibitem{ALLPT} R. Adamczak R. Lata{\l}a, A. Litvak, A. Pajor and N. Tomczak-Jaegermann,
in preparation.

\bibitem{BN} S.G. Bobkov and F.L. Nazarov, 
\emph{On convex bodies and log-concave probability measures with unconditional basis},
in: Geometric aspects of functional analysis, Lecture Notes in Math. 1807, Springer, Berlin, 2003, 53--69.

\bibitem{Bo} C. Borell, \emph{Convex measures on locally convex spaces}, Ark. Math. 12
(1974), 239--252.

\bibitem{DMS} S.J Dilworth and S.J. Montgomery-Smith, 
\emph{The distribution of vector-valued Rademacher series}
Ann. Probab. 21 (1993), 2046--2052. 

\bibitem{LW} R. Lata{\l}a and J.O. Wojtaszczyk, \emph{On the infimum convolution inequality}, 
Studia Math. 189 (2008), 147--187. 

\bibitem{L3} R. Lata{\l}a, {\em Tail and moment estimates for sums of independent 
random vectors with logarithmically concave tails}, Studia Math. 118 (1996), 301--304. 

\bibitem{L2} R. Lata{\l}a, \emph{On weak tail domination of random vectors}, Bull. Polish Acad. Sci. Math. 57 (2009), 75--80. 

\bibitem{LT} M. Ledoux and M. Talagrand, \emph{Probability in Banach spaces. Isoperimetry and processes},
Springer-Verlag, Berlin, 1991.

\bibitem{MP} B. Maurey and G. Pisier, \emph{S\'eries de variables aléatoires vectorielles ind\'ependantes et propri\'et\'es g\'eom\'etriques des espaces de Banach}, Studia Math. 58 (1976), 45--90.

\bibitem{MS} V.D. Milman and G. Schechtman, \emph{Asymptotic theory of finite-dimensional normed spaces},
Lecture Notes in Math. 1200, Springer-Verlag, Berlin, 1986.

\bibitem{Pa} G. Paouris, \emph{Concentration of mass on convex bodies}, 
Geom. Funct. Anal. 16 (2006), 1021--1049.

\bibitem{Ta1}  M. Talagrand, \emph{A new isoperimetric inequality and the concentration of measure phenomenon}, in: Israel Seminar (GAFA), Lecture Notes in Math. 1469, Springer, Berlin, 1991, 94--124.

\bibitem{Ta2} M. Talagrand \emph{The supremum of some canonical processes}, Amer. J. Math. 116 (1994), 
283--325. 

\bibitem{Ta} M. Talagrand, \emph{The generic chaining. Upper and lower bounds of 
stochastic processes}, Springer, Berlin, 2005.

\end{thebibliography}
\end{document}